\documentclass{amsart}
\usepackage{enumerate}
\usepackage{calc}
\usepackage[dvips]{graphicx}
\usepackage{color}
\usepackage{pstricks}
\usepackage{epigraph}
\usepackage{amsmath, amssymb}
\usepackage{pst-node}
\newtheorem{theo}{Theorem}[section]

\newtheorem{cor}[theo]{Corollaire}

\newtheorem{lem}[theo]{Lemma}

\newtheorem{prop}[theo]{Proposition}

\newtheorem{defn}[theo]{Definition}
\newtheorem{hyps}[theo]{Assumptions}
\newtheorem{hyp}[theo]{Assumption}

\theoremstyle{definition}
\newtheorem{rem}[theo]{Remark}

\def\a{{\a}}
\def\a{{\mathfrak a}}

\def\C{\mathbb C}

\def\E{\mathbb E}

\def\N{\mathbb N}
\def\P{\mathbb{P}}

\def\R{\mathbb R}

\def\Z{\mathbb{Z}}

\DeclareMathOperator{\ch}{ch}
\begin{document}  
\title[Brownian paths  in  an alcove ]{Brownian paths in an alcove and the Littelmann path model\\ {\scriptsize what we know what we do not  know what we hope}}

\author{Manon Defosseux}
\maketitle  
 \begin{abstract}  We present some results  connecting   Littelmann paths and Brownian paths in the framework of affine Kac--Moody algebras.  We prove in particular that the string coordinates associated to a specific sequence of random Littelmann paths converge towards their analogs for Brownian paths.   At the end we explain why we hope that our results  will be the first steps on a way which could hopefully lead to a Pitman type theorem for a Brownian motion in an alcove associated to an affine Weyl group. 
 \end{abstract}

\section{Introduction}
A Pitman's theorem states that if $\{b_t,t\ge 0\}$ is a one dimensional Brownian motion, then $$\{b_t - 2 \inf_{0 \leq s\leq t}b_s, t\ge 0\}$$ is a three dimensional Bessel process, i.e.\ a Brownian motion conditioned  in Doob's sense  to remain forever positive \cite{pitman}.  Philippe Biane, Philippe Bougerol and Neil O'Connell have proved in \cite{bbo2} that a similar theorem exists in which the real Brownian motion is replaced by a Brownian motion on a finite dimensional real vector space. A finite Coxeter group acts on this space and the positive Brownian motion is replaced by a conditioned Brownian motion with values in a fundamental domain for the action of this group. In that case, the second process   is obtained by applying to the first one  Pitman   transformations in successive directions given by a reduced decomposition of the longest word in the Coxeter group. 

Paper  \cite{boubou-defo}  gives a similar representation  theorem for a space-time real Brownian motion $\{(t,b_t):t\ge 0\}$  conditioned to remain in  the cone
$$\mathcal C'=\{(t,x)\in \R\times\R: 0\le x\le t\}.$$
Actually $\mathcal C'$ is a fundamental domain for the action on $\R_+\times \R$ of an affine Coxeter group of type $A^1_1$. This affine Coxeter group, which is not a finite group, is generated by two reflections and it could be natural to think  that one could obtain a space-time Brownian motion conditioned to remain in $\mathcal C'$ applying successively and infinitely to a space-time Brownian motion two Pitman transformations corresponding to these two reflections. We have proved with Philippe Bougerol in \cite{boubou-defo} that this is not the case. Actually a L\'evy type transformation has to be added at the end of the successive Pitman transformations if we want to get a Pitman's   representation theorem in this case.  

It is now  natural  to ask if  such a theorem exists for the other affine Coxeter groups. We will focus on Coxeter groups  of type $A_n^1$, with $n\ge 1$. Such a Coxeter group  is the Weyl group of a type $A$ extended  affine Kac--Moody algebra.  The presence of the L\'evy transformations in the case when  $n=1$   makes   the higher rank statement quite open. 

When $n=1$, the proof of the Pitman type theorem in  \cite{boubou-defo} rests on an approximation of the affine Coxeter group   by a sequence of dihedral groups for which the results of \cite{bbo2} are applicable. Such an approximation does not exist  for the higher ranks.  Nevertheless,   another approximation exists of the Brownian model   that we are interested in. It involves the Littelmann path model.  

The Littelmann's model  is a combinatorial model which allows to describe weight properties of some particular integrable representations of Kac--Moody algebras.   Philippe Biane, Philippe Bougerol and Neil O'Connell pointed out in \cite{bbo}  the fundamental fact that the Pitman transformations are intimately related to the Littelmann path model.  In the case of an affine Lie algebra, this model allows to construct random paths   which approximates the Brownian model that we are interested in.   The knowledge  of the Littelmann paths properties gives then a way to get a better understanding of those of the Brownian paths.     

Unfortunately this approach didn't lead for the moment to  a Pitman type theorem in an affine framework. Nevertheless, we have obtained  several non trivial results that we present here. In particular,  we prove the convergence of the string coordinates arising in the framework of the Littelmann path model towards their analogs defined for the Brownian paths.  Besides we  use the Littelmann path approach to try to guess which correction could be needed if a Pitman type theorem existed in this case.  We propose here a conjectural correction, with   encouraging simulations. 

Finally, notice  that   the space component of a  space-time Brownian motion conditioned to remain in an affine Weyl chamber is equal, up to a time inversion, to a Brownian motion conditioned to remain in an alcove, so that our suggestion provides also a suggestion for    a Pitman's theorem for this last conditioned process.

These notes are organized  as follows. In section \ref{BD} we recall the necessary background about affine Lie algebras and their representations.   The Littelmann path model in this context  is explained briefly in section \ref{LPM}   where we recall  in particular the definition of the string coordinates. In section  \ref{RWLP} we define two sequences of random Littelmann paths. The first one converges towards a space-time Brownian motion in the dual of a Cartan subalgebra of an affine Lie algebra. The second one converges towards a space-time Brownian motion in an affine Weyl chamber. This last process is defined in section \ref{CSTBM}.  The statements of the two convergences are given in section  \ref{WWK}, where    we also prove the convergence of the string coordinates associated to a sequence of  random Littelmann paths towards their analogs for Brownian paths. Section \ref{WWDNK} is devoted to explain what is missing in the perspective of a Pitman type theorem. Finally in section \ref{WWH}  we use the description of the highest weight Littelmann modules given in \cite{littel2}  to suggest  transformations which could play the role of   L\'evy transformations in the case when $n$ is greater than one.

\section{Basic definitions}\label{BD}
This section is based on  \cite{Kac}. In order to make the paper as easy as possible to read, we consider only the case of an extended affine Lie algebra of type $A$. For this we consider a realization $(\widehat{\mathfrak h},\widehat\Pi,\widehat\Pi^\vee)$ of a Cartan matrix of type $A^{1}_n$ for $n\ge 1$. That is to say
$$\widehat\Pi=\{\alpha_0,\dots,\alpha_n\}\subset \widehat{\mathfrak h}^*\textrm{ and } \widehat\Pi^\vee=\{ \alpha_0^\vee,\dots,\alpha_n^\vee\}\subset \widehat{\mathfrak h} $$ with
$$\langle \alpha_i,\alpha_j^\vee \rangle=\left\{
    \begin{array}{ll}
      2&\textrm{ if $i=j$}\\ 
-1&\textrm{ if $\vert i-j\vert \in\{1,n\}$ when $n\ge 2$ }\\
-2&\textrm{ if $\vert i-j\vert =1$ when $n=1$,}\\
    \end{array}
\right.$$
and $\dim \widehat{\mathfrak h}=n+2$, where $\langle \cdot,\cdot\rangle$ is the canonical pairing. We consider an element $d\in\widehat{ \mathfrak h}$ such that $$\langle\alpha_i,d\rangle= \delta_{i0},$$ for $ i\in \{0,\dots,n\}$ and define $\Lambda_0\in \widehat{\mathfrak h}^*$ by
$$ \langle \Lambda_0,d\rangle =0 \textrm{ and }   \langle \Lambda_0,\alpha^\vee_i\rangle =\delta_{i0},$$
for $ i\in \{0,\dots,n\}$. We consider the Weyl group $\widehat W$ which is  the subgroup of $ \mbox{GL}(\widehat{\mathfrak{h}}^*)$ generated by the simple reflexions $s_{\alpha_i}$, $i\in \{0,\dots,n\}$, defined by
$$s_{\alpha_i}(x)=x-\langle x,\alpha_i^\vee\rangle\alpha_i, \quad x\in \widehat{\mathfrak h}^*,$$
$i\in\{0,\dots,n\}$, and equip $\widehat{\mathfrak h}^*$ with a non degenerate $\widehat{W}$-invariant bilinear form $(\cdot\vert\cdot)$ defined by 
$$\left\{
  \begin{array}{ll}
(\alpha_i\vert\alpha_j)=  2&\textrm{ if $i=j$ } \\ 
(\alpha_i\vert\alpha_j)=-1&\textrm{ if $\vert i-j\vert \in\{1,n\}$ when $n\ge 2$ }\\
(\alpha_i\vert\alpha_j)=-2&\textrm{ if $\vert i-j\vert =1$ when $n=1$,}\\
(\alpha_i\vert\alpha_j)=0&\textrm{ otherwise,}
  \end{array}
\right.$$
 $(\Lambda_0\vert\Lambda_0)=0$ and $(\alpha_i\vert\Lambda_0)=\delta_{i0}$, $i\in\{0,\dots,n\}$.
 We consider as usual the set of integral weights $$\widehat P=\{\lambda\in \widehat{\mathfrak h}^* : \langle \lambda,\alpha_i^\vee\rangle \in \Z, i\in \{0,\dots,n\}\},$$
and the set of dominant integral weights $$\widehat P^+=\{\lambda\in\widehat{ \mathfrak h}^* : \langle \lambda,\alpha_i^\vee\rangle \in \N, i\in \{0,\dots,n\}\}.$$ For $\lambda\in \widehat P_+$ we denote by $V(\lambda)$ the irreducible highest  weight module of highest weight $\lambda$ of an affine Lie algebra  of type  $A_n^{1}$ with $\widehat{\mathfrak h}$ as a Cartan subalgebra and $\widehat \Pi$ as a set of simple roots.   We consider the formal character  
$$\mbox{ch}_\lambda =\sum_{\beta\in \widehat P} m_\beta^\lambda e^\beta,$$
where $m_\beta^\lambda $ is the multiplicity of the weight $\beta$ in $V(\lambda)$. If $\widehat \nu\in\widehat{\mathfrak h}^*$  satisfies $(\widehat\nu\vert \alpha_0)>0$ then the series 
$$ \sum_{\beta\in \widehat P}  m_\beta^\lambda e^{(\beta\vert\widehat\nu)},$$
converges and we denote by $\mbox{ch}_\lambda(\widehat\nu)$ its limit.  

\section{Littelmann path model}\label{LPM}
From now on we work on the   real vector space   
$$\widehat{\mathfrak h}^*_\R=\R\Lambda_0\oplus \bigoplus_{i=0}^n\R\alpha_i.$$
In this section, we recall what we need about the Littelmann path model (see mainly \cite{littel2} for more details, and also  \cite{littel1}). Fix $T\ge 0$. A path $\pi$ is a piecewise linear function $\pi:[0,T]\to \widehat{\mathfrak h}_\R^*$ such that $\pi(0)=0$. We consider the cone generated by $\widehat P_+$  $$\mathcal C=\{\lambda\in \widehat{\mathfrak h}^*_\R: \langle\lambda,\alpha_i^\vee\rangle\ge 0, i\in \{0,\dots,n\}\}.$$ A path $\pi$ is called dominant if $\pi(t)\in \mathcal C$ for all $t\in [0,T]$. It is called integral if $\pi(T)\in \widehat P$ and $$\min_{t\in [0,T]}\langle\pi(t),\alpha_i^\vee\rangle\in \Z,\textrm{  for all } i\in \{0,\dots,n\}.$$ 
\paragraph{\bf Pitman's transforms, Littelmann module.} We define the Pitman's transforms $\mathcal P_{\alpha_i}$,  $i\in \{0,\dots,n\}$,  which operate on the set of continuous functions $\eta : \R_+ \to \widehat{\mathfrak h}^*_\R$  such that $\eta(0) = 0$.  They are given by 
$$\mathcal P_{\alpha_i}\eta(t)=\eta(t)-\inf_{s\le t}\langle\eta(s),\alpha_i^\vee\rangle\alpha_i, \quad t\in \R_+,$$
$i\in \{0,\dots,n\}$. Let  us fix a sequence  $(i_k)_{k\ge 0}$   with values in $\{0,\dots,n\}$ such that 
\begin{align}\label{condition} 
\vert \{k:i_k=j\}\vert =\infty\, \textrm{ for all $j\in \{0,\dots,n\}.$}
\end{align} 
Given an integral dominant path   $\pi$ defined on $[0,T]$, such that $\pi(T)\in \widehat P_+$, the Littelmann  module $B\pi$ generated by $\pi$ is the set of paths   $\eta$ defined on  $[0,T]$ such that it exists  $k\in \N$ such that $$\mathcal P_{\alpha_{i_k}}\dots\mathcal P_{\alpha_{i_0}}\eta=\pi.$$  This  module doesn't  depend on the sequence $(i_k)_{k\ge 0}$ provided that it satisfies condition (\ref{condition}).
 
 For an integral dominant path $\pi$, one defines $\mathcal{P}$ on $(B\pi)^{\star \infty}$, where $\star$ stands for the usual concatenation (not the Littelmann's one), letting    for $ \eta\in (B\pi)^{\star \infty}$
 $$\mathcal{P}\eta(t)=\lim_{k\to \infty}\mathcal P_{\alpha_{i_k}}\dots\mathcal P_{\alpha_{i_0}}\eta(t), \quad t\ge 0.$$
  Note that for any $u\in \R_+$ it  exists $k_0$ such that
$$\mathcal P\eta(t)=\mathcal P_{\alpha_{i_k}}\dots\mathcal P_{\alpha_{i_0}}\eta(t), \quad t\in[0,u]\footnote{Notice that this fact remains true if $\eta$  has a piecewise $C^1$ component in $\mathfrak h_\R^*$.} ,$$ and that the definition of $\mathcal P$ does not depend on the order in which  the Pitman's transfoms are applied provided that each of them is possibly applied infinitely many times. \\

 \paragraph{\bf String coordinates, Littelmann transforms.}  In the following, the sequence  $(i_k)_{k\ge 0}$ can't be chosen arbitrarily. It is important that 
 \begin{align} \label{cond-red} 
 \textrm{for each $k\ge 0$, }s_{\alpha_{i_k}}\dots s_{\alpha_{i_{1}}}s_{\alpha_{i_{0}}} \textrm{ is a reduced decomposition}.
 \end{align}
   It is the case for instance   if $i_k=k \mod(n+1)$, for every $k\ge 0$.    From now  on we  fix  a sequence $\mathfrak i=(i_k)_{k\ge 0}$ such that (\ref{cond-red}) is satisfied (then condition (\ref{condition}) is also satisfied). For  a dominant path $\pi$ defined on $[0,T]$, we consider the  application $\mathfrak{a^i}$ from $B\pi$ to the set of almost zero  nonnegative integer sequences $\ell^{(\infty)}(\N)$ such that for $\eta\in B\pi$, $\mathfrak{a^i}(\eta)$ is the sequence of integers $(a^{\mathfrak i}_k)_{k\ge 0}$ in $\ell^{(\infty)}(\N)$ defined by the identities 
\begin{align}\label{def-string} 
\mathcal P_{\alpha_{i_m}}\dots\mathcal P_{\alpha_{i_0}}\eta(T)  =\eta(T)+\sum_{k=0}^ma^{\mathfrak i}_k\alpha_{i_k},\quad m\ge 0.
\end{align}
Notice that we will  most often omit $\mathfrak i$ in  $\mathfrak{a^i}$ and $a^\mathfrak{i}_k$, $k\ge 0$. Let us now give the connection with the Littelmann model described in \cite{littel2}. We consider   
$$w^{(p)}=s_{\alpha_{i_p}}\dots s_{\alpha_{i_{1}}}s_{\alpha_{i_{0}}} ,$$  for any $p\ge 1.$ 
Notice that the reflexions are not labeled in the same order as in \cite{littel2}. Nevertheless the path operators $e_{i_k}$ and $f_{i_k}$ defined in \cite{littel2} are     applied in the same order here.  For a tuple $\underline{a} = (a_0,\dots ,a_p)\in \N^{p+1}$  we write   $f^{\underline a}$ for
$$f^{\underline a} :=f^{a_0}_{i_0}\dots f^{a_p}_{i_p}.$$  For a dominant integral path $\pi$ and $\eta\in   B\pi$ such that   $\eta=f^{\underline a}\pi$ for $\underline a=(a_0,\dots,a_p),$  one says that $\underline{a}$ is an adapted string for $\eta$ if $a_0$ is the largest integer such that   
$e_{i_0}^{a_0}\eta\ne 0$, $a_1$ is the largest integer such that  
$e_{i_1}^{a_1}e_{i_0}^{a_0}\eta\ne 0$ and so on.  Actually, given $\eta$, the integers $a_0,a_1,\dots,$ are exactly the ones defined by (\ref{def-string}). In particular the application $\mathfrak{a^i}$ is injective on $B\pi$. Peter Littelmann   describes its image in \cite{littel2}. For this he defines $\mathcal S_{w^{(p)}}$ as the set of all $\underline{a}\in \N^{p+1}$ such that $\underline{a}$ is an adapted string  of $f^{\underline{a}}\pi$ for some dominant integral path $\pi$ and $\mathcal S_{w^{(p)}}^\lambda$ as the subset $\{\underline a\in \mathcal S_{w^{(p)}} : f^{\underline a}\pi_\lambda\ne 0\}$ where $\pi_\lambda$ is a dominant integral path ending at $\lambda\in P_+$. The set $\mathcal S_{w^{(p)}}^\lambda$ can be identified with the vertices of the crystal graph of a Demazure module. It  depends on $\pi_\lambda$ only throught $\lambda$.  If we let $$B(\infty)=\bigcup_{p\in \N} \mathcal S_{w^{(p)}}\quad and \quad B(\lambda)=\bigcup_{p\in \N}\mathcal S_{w^{(p)}}^\lambda,$$  proposition 1.5 of \cite{littel2} gives the following one, which will be essential to try to guess what   the L\'evy transformations could be for $n\ge 2$. 
\begin{prop} \label{Cones}
\begin{align*}
B(\lambda)&=\{\underline{a}\in B(\infty) : a_p\le \langle \lambda -\sum_{k=p+1}^\infty a_k\alpha_{i_k},\alpha_{i_p}^\vee\rangle, \forall p \ge 0\}\\
&=\{\underline{a}\in B(\infty) : a_p\le \langle \lambda -\omega(\underline{a})+\sum_{k=0}^p a_k\alpha_{i_k},\alpha_{i_p}^\vee\rangle, \forall p \ge 0\}\\
&=\{\underline{a}\in B(\infty) : \langle   \omega(\underline{a})-\sum_{k=0}^{p-1} a_k\alpha_{i_k}-\frac{1}{2}a_{p}\alpha_{i_p},\alpha_{i_p}^\vee\rangle\le \langle \lambda,\alpha_{i_p}^\vee\rangle, \forall p \ge 1\},
\end{align*}
where $\omega(\underline a)=\sum_{k=0}^\infty a_k\alpha_{i_k}$, which is the opposite of the weight of $\underline a$ in the crystal $B(\infty)$  of the Verma module of highest weight $0$.
\end{prop}

\section{Random walks and Littelmann paths}  \label{RWLP}
Let us consider a path $\pi_{\Lambda_0}$ defined on $[0,1]$ by 
$$\pi_{\Lambda_0}(t)=t\Lambda_0, \quad t\in [0,1],$$
and the Littelmann module $B{\pi_{\Lambda_0}}$ generated by $\pi_{\Lambda_0}$. We fix an integer $m\ge 1$, choose $\widehat \nu \in\widehat{ \mathfrak h}_\R^*$ such that $(\alpha_0\vert\widehat \nu)>0$. Littelmann path theory ensures that 
$$\mbox{ch}_{\Lambda_0}(\widehat \nu/m)=\sum_{\eta\in  B{\pi_{\Lambda_0}}}e^{\frac{1}{m}(\eta(1)\vert\widehat \nu)}.$$ 
We equip $B{\pi_{\Lambda_0}}$ with a probability measure $\mu^{m}$ letting 
\begin{align}\mu^{m}(\eta)=\frac{e^{\frac{1}{m}(\eta(1)\vert\widehat \nu)}}{\mbox{ch}_{\Lambda_0}(\widehat \nu/m)}, \quad \eta\in B{\pi_{\Lambda_0}}.
\end{align}\label{loi-discrete}
One considers a sequence  $(\eta_i^{m })_{i\ge 0}$ of i.i.d random variables with law $\mu^{m}$ and   a random path $\{\pi^{m}(t),t\ge 0\}$  defined by
$$\pi^{m}(t)=\eta_1^{m}(1)+\dots+\eta_{k-1}^{m}(1)+ \eta_k^{m}(t-k+1),$$ when $t\in [k-1,k[, $ for $k\in \Z_+$. 
 The Littelmann's path theory implies immediately the following proposition. 

\begin{prop} \label{DMC} The random process $\{\mathcal P(\pi^{m})(k),k\ge 0\}$ is a Markov chain starting from $0$ with values in $\widehat P_+$ and transition probability 
$$Q(\lambda,\beta)=\frac{\ch_\beta(\widehat \nu/m)}{\ch_\lambda(\widehat \nu/m)\ch_{\Lambda_0}(\widehat \nu/m)}M_{\lambda,\Lambda_0}^{\beta}, \quad \lambda,\beta\in \widehat P_+,$$
where $M_{\lambda,\Lambda_0}^{\beta}$ is the number of irreducible   representations in the isotypic componant  of  highest  weight $\beta$ in $V(\lambda)\otimes V(\Lambda_0)$.  
\end{prop}

\begin{rem} If $\delta$ is the lowest positive null root, i.e. $\delta=\sum_{i=0}^n\alpha_i$, then $V(\lambda)$ and $V(\beta)$ are isomorphic for $\lambda=\beta \mod \delta$. Thus  $\{\mathcal P(\pi^{m})(k),k\ge 0\}$ remains markovian in  the quotient space $\widehat{\mathfrak h}_\R^*/\R\delta$. This is this process that interests us.\end{rem}

\section{A conditioned space-time Brownian motion}\label{CSTBM} 
One considers  the decomposition  $$\widehat{\mathfrak h}_\R^*=\R\Lambda_0\oplus\mathfrak{h}^*_\R \oplus \R\delta$$  where $\mathfrak{h}^*_\R= \bigoplus_{i=1}^n \R\alpha_i$ and one identifies $\widehat{\mathfrak h}_\R^*/\R\delta$  with $\R\Lambda_0\oplus\mathfrak{h}^*_\R $. We let $\mathfrak h^*= \bigoplus_{i=1}^n \C\alpha_i$.  We denote by $R_+$ the set of positive roots in $\mathfrak h^*$, by $\rho$  the corresponding Weyl vector, i.e. $\rho=\frac{1}{2}\sum_{\alpha\in R_+}\alpha$, by $P_+$ the corresponding set of dominant weights, i.e. 
$$P_+ =\{\lambda\in \mathfrak{h}_\R^* : \langle \lambda,\alpha_i^\vee\rangle \in \N \textrm{ for } i\in \{1,\dots,n\}\},$$  and by $W$ the subgroup of $\widehat W$ generated by the simple reflexions $s_{\alpha_i}$, $i\in\{1,\dots,n\}$.    The bilinear form $(\cdot\vert \cdot)$ defines a scalar product $\mathfrak h_\R^*$ so that we write $\vert \vert x\vert\vert^2$ for  $(x\vert x)$ when $x\in \mathfrak h_\R^*$. One  considers  a standard Brownian motion $\{b_t:t\ge 0\}$ in $\mathfrak{h}^*_\R$ with drift $\nu\in \mathfrak h_\R^*$ and $\{B_t=t\Lambda_0+b_t: t\ge 0 \}$, which is a space-time Brownian motion.  We define the function $\pi$ on  $ \widehat{\mathfrak{h}}^*$  letting  for $x\in \widehat{\mathfrak{h}}^*$,
 $$\pi(x)=\prod_{\alpha\in R_+}\sin\pi(\alpha\vert x),$$ and for $\lambda_1, \lambda_2\in \R_+^*\Lambda_0+\mathfrak{h}^*_\R $,  we let $$\psi_{\lambda_1}(\lambda_2)=\frac{1}{\pi(\lambda_1/t_1)}\sum_{w\in \widehat W}\det(w)e^{(w\lambda_1\vert\lambda_2)}.$$
where  $t_1=(\delta\vert\lambda_1)$.
Using a Poisson summation formula, Igor  Frenkel has proved in \cite{fre} (see also \cite{defo3}) that  if $\lambda_1=t_1\Lambda_0+x_1$ and $\lambda_2=t_2\Lambda_0+x_2$, for $x_1,x_2\in \mathfrak h_\R^*$, then
$\frac{\psi_{\lambda_1}(\lambda_2)}{\pi(\lambda_2/t_2) }$ is proportionnal to 
\begin{align}\label{form-fre}
(\frac{t_1t_2}{2\pi})^{-n/2}e^{\frac{t_1}{2t_2}(\lambda_2,\lambda_2)+\frac{t_2}{2t_1}(\lambda_1,\lambda_1)}\sum_{\mu\in P_+}\chi_\mu(x_2/t_2)\chi_\mu(-x_1/t_1)e^{-\frac{1}{2t_1 t_2}(2\pi)^2\vert \vert \mu+\rho\vert \vert^2},
\end{align}
where    $\chi_\mu$ is the character of the representation of highest weight $\mu$ of the underlying semi-simple Lie algebra having $\mathfrak h$ as a Cartan subalgebra, i.e.  $$\chi_\mu(\beta)\propto \frac{1}{\pi(\beta)}\sum_{w\in W}e^{2i\pi (w(\mu+\rho)\vert \beta)}, \, \beta\in \mathfrak{h}_\R^*.$$
We consider the cone $\mathcal C'$ which is  the cone $\mathcal C$ viewed in the quotient space, i.e.   $$\mathcal C'=\{\lambda\in\widehat{ \mathfrak h}_\R^*/\R\delta: \langle\lambda,\alpha_i^\vee\rangle\ge 0, i=0,\dots,n\},$$
and the stopping time $T=\inf\{t\ge 0:   B_t\notin \mathcal C'\}$.    One recognizes in the sum over $P_+$ of (\ref{form-fre}), up to a positive multiplicative constant, a probability density related to the heat Kernel on a compact Lie group (see \cite{fre} or \cite{defo3} for more details). As the function $\pi$ is a positive function on the interior of the cone $\mathcal C'$ vanishing on its boundary, one obtains easily the following proposition.  We let $\widehat \nu=\Lambda_0+\nu$.
\begin{prop} The function $$\Psi_{\widehat \nu}: \lambda\in \mathcal C'\to e^{-(\widehat{\nu}\vert \lambda)}\psi_{\widehat{\nu}}(\lambda)$$  is  a  constant sign  harmonic function for the killed process $\{  B_{t\wedge T}: t\ge 0\}$, vanishing only on the boundary of  $\mathcal C'$. 
\end{prop}
 
\begin{defn}
We define $\{A_t: t\ge 0\}$ as the killed process $\{  B_{t\wedge T}: t\ge 0\}$ conditionned (in Doob's sens) not to die, via the harmonic function 
$\Psi_{\widehat{\nu}}$.
\end{defn}
\section{What we know}\label{WWK}
From now on we suppose  that $\widehat \nu=\Lambda_0+\nu$ in (\ref{loi-discrete}), where $\nu \in \mathfrak h^*_\R$ such that $\widehat \nu\in \mathcal C'$.
\subsection{Convergence of the random walk and the Markov chain}
The Fourier transform of $\eta_1^{m}(1)$ can be written with the character $\mbox{ch}_{\Lambda_0}$ and the following lemma is easily obtained using a Weyl character formula and formula (\ref{form-fre}).   
\begin{lem} \label{conv1}   For $x\in \mathfrak h^*$, one has
  $$\lim_{m\to \infty} \E(e^{(x\vert\frac{1}{m}\sum_{i=1}^{[mt]}\eta_i^{m}(1))})=e^{\frac{t}{2}(( x+\nu\vert x+\nu) -(\nu\vert\nu))}.$$ 
 
\end{lem}
As the coordinate  of $\sum_{i=1}^{[mt]}\eta_i^{m}(1)$ along $\Lambda_0$ is $[mt]$, the previous lemma shows that   the random walk whose increments are distributed according to $\mu^{m}$ converges in $\widehat{\mathfrak h}_\R^*/\R\delta$  after a  renormalisation in $1/m$ towards a space-time Brownian motion, the time component being along $\Lambda_0$. By analycity, lemma implies  also the convergence of the joint moments.
\begin{lem}  The joint moments of $\frac{1}{m}\sum_{i=1}^{[mt]}\eta_i^{m}(1)$ converge in the quotient space towards the ones of $B_t$.
\end{lem}
 One can  show as in \cite{defo2} that in the quotient space the Markov chain of proposition \ref{DMC} converges also. One has the following proposition where  convergences  are convergences in finite-dimensional distribution. We denote by $\lceil .\rceil$ the ceiling function.
\begin{prop} \label{conv} In the quotient space $\widehat{\mathfrak h}^*_\R/\R\delta$ one has the following convergences.
\begin{enumerate}  
\item The sequence $$\{\frac{1}{m}\pi^{m}(\lceil mt\rceil): t\ge 0 \}, \quad m\ge 1,$$  converges towards $\{B_t: t\ge 0\}$ when $m$ goes to infinity.
\item  The sequence $$\{\frac{1}{m}\mathcal P\pi^{m}(\lceil mt\rceil): t\ge 0\}, \quad m\ge 1,$$ converges towards $\{A_t: t\ge 0\}$ when $m$ goes to infinity.
\end{enumerate}
\end{prop}
Now we want to prove that the first sequence of the proposition is tight in order to prove that the string coordinates associated to the Littelmann path model converges towards their analogs defined for the Brownian paths.
\subsection{Convergence of the string coordinates}
We want to prove that for any  integer $k$
$$\{\frac{1}{m}\mathcal P_{\alpha_{i_k}}\dots\mathcal P_{\alpha_{i_0}}\pi^{m}(\lceil mt\rceil):t\ge 0\} \textrm{ and } \{\frac{1}{m}\mathcal P_{\alpha_{i_k}}\dots\mathcal P_{\alpha_{i_0}}\pi^{m}(mt):t\ge 0\}$$
converges  in the quotient space  $\widehat{\mathfrak h}^*_\R/\R\delta$ towards 
$$\{\mathcal P_{\alpha_{i_k}}\dots\mathcal P_{\alpha_{i_0}}B(t):t\ge 0\},$$
when $m$ goes to infinity.
For this we will prove that the sequence 
$\{\frac{1}{m}\pi^{m}(mt): t\ge 0\}$, $m\ge 1$, is tight. We begin to establish the   following lemmas, which will be used to control the increments.
\begin{lem} \label{conv2}   For $x\in \mathfrak h^*$, one has
 $$\lim_{m\to \infty} \E(e^{(x\vert \frac{1}{\sqrt{m}}\eta_1^{m}(1))})=e^{\frac{1}{2}(x\vert x)}$$  
 and the joint moments of the projection  of $\frac{1}{\sqrt{m}}\eta_1^m(1)$ on $\mathfrak{h}_\R^*$ converge towards the ones of a standard Gaussian  random variable on $\mathfrak{h}_\R^*$.
\end{lem}
\begin{proof} 
We use the character formula and  formula (\ref{form-fre}) for the first convergence and analycity for the convergence of the moments.
\end{proof}
In   lemma \ref{condlaw} and its corollary, for a  path $\eta$ defined on $[0,1]$, $a_0(\eta)$ is the first string coordinate  of $\eta$ corresponding to the sequence $\mathfrak i=i_0,i_1,\dots$, that is to say $a_0(\eta)=-\inf_{s\le 1}\langle\eta(s),\alpha_{i_0}^\vee\rangle$.
\begin{lem} \label{condlaw} For $r\in \N$, $u\in \C$,
$$\E(e^{(u\alpha_{i_0}\vert \eta_1^{m}(1))}\vert \langle\eta_1^{m}(1)+a_0(\eta_1^{m})\alpha_{i_0},\alpha_{i_0}^\vee\rangle=r)=\frac{s_r(e^{\frac{1}{2}{(\alpha_{i_0}\vert u\alpha_{i_0}+\widehat\nu/m)}})}{s_r(e^{\frac{1}{2}(\alpha_{i_0}\vert\widehat\nu/m)})},$$
where $s_r(q)=\frac{q^{r+1}-q^{-(r+1)}}{q-q^{-1}}$, $q>0$.
\end{lem}
\begin{proof}  We use the description of the cone of string coordinates in proposition \ref{Cones} and the fact that there is no condition for the first string coordinate  in    $B(\infty)$. \end{proof}
As $\langle \eta_1^{m}(1), \alpha_{i_0}^\vee\rangle$ admits a Laplace transform defined on $\R$, which implies in particular that the moments of each odrer of $\langle \eta_1^{m}(1), \alpha_{i_0}^\vee\rangle$ exist, the lemma has the following corollary.
\begin{cor} $$\E(\langle\eta_1^{m}(1),\alpha^\vee_{i_0}\rangle^4)=\E([ \langle\eta_1^{m}(1)+a_0(\eta_1^{m})\alpha_{i_0},\alpha_{i_0}^\vee\rangle+1]^4)$$
\end{cor}
As the Littelmann module $B\pi$ does not depend on the sequence $\mathfrak i=(i_k)_{k\ge 0}$, corollary implies immediately the following proposition.
\begin{prop} \label{maj} For all $i\in\{0,\dots,n\}$, 
$$\E(\langle\eta_1^{m}(1)-\inf_{s\le 1}\langle\eta_1^{m}(s),\alpha_i^\vee\rangle\alpha_i,\alpha_i ^\vee\rangle^4)\le \E(\langle\eta^{m}_1(1),\alpha_i^\vee\rangle^4)$$
\end{prop} 

\begin{prop}\label{boundinc}  One has for any $i\in\{0,\dots,n\}$, $m\ge 1$, $t\ge 0$,
$$\vert\langle\pi^{m}(mt)-\pi^{(n)}(\lceil mt\rceil),\alpha_i^\vee\rangle\vert \le   \sum_{j=0  }^n\langle\eta^{m}_{\lceil mt\rceil}(1)-\inf_{s\le 1}\langle\eta^{m}_{\lceil mt\rceil}(s),\alpha_j^\vee\rangle\alpha_j,\alpha_j^\vee\rangle$$
\end{prop}
\begin{proof}
\begin{align*}
\vert\langle\pi^{m}(mt)&-\pi^{m}(\lceil mt\rceil),\alpha_i^\vee\rangle\vert\\
&\le \max(\langle\eta^{m}_{\lceil mt\rceil}(1),\alpha_i^\vee\rangle-\inf_{s\le 1}\langle\eta^{m}_{\lceil mt\rceil}(s),\alpha_i^\vee\rangle,\sup_{s\le 1}\langle\eta^{m}_{\lceil mt\rceil}(s),\alpha_i^\vee\rangle-\langle\eta^{m}_{\lceil nt\rceil}(1),\alpha_i^\vee\rangle)
\end{align*} Besides 
$$\langle\eta^{m}_{\lceil mt\rceil}(1),\alpha_i^\vee\rangle-\inf_{s\le 1}\langle\eta^{m}_{\lceil mt\rceil}(s),\alpha_i^\vee\rangle\le \langle\eta^{(n)}_{\lceil mt\rceil}(1)-\inf_{s\le 1}\langle\eta^{(n)}_{\lceil mt\rceil}(s),\alpha_i^\vee\rangle\alpha_i,\alpha_i^\vee\rangle$$
and 
\begin{align*}
\sup_{s\le 1}\langle\eta^{m}_{\lceil mt\rceil}(s),\alpha_i^\vee\rangle-\langle\eta^{m}_{\lceil nt\rceil}(1),\alpha_i^\vee\rangle&=\sup_{s\le 1}(\delta-\sum_{j\ne i}\alpha_j\vert\eta^{m}_{\lceil mt\rceil}(s))-(\delta-\sum_{j\ne i}\alpha_j\vert\eta^{m}_{\lceil mt\rceil}(1))\\
&=\sup_{s\le 1}(s-\sum_{j\ne i}(\alpha_j\vert\eta^{m}_{\lceil mt\rceil}(s)))-(1-\sum_{j\ne i}(\alpha_j\vert\eta^{m}_{\lceil mt\rceil}(1)))\\
&\le \sum_{j\ne i}(\alpha_j\vert\eta^{m}_{\lceil mt\rceil}(1))-\inf_{s\le 1}((\alpha_j\vert\eta^{m}_{\lceil mt\rceil}(s))\\
&\le \sum_{j\ne i} \langle\eta^{m}_{\lceil mt\rceil}(1)-\inf_{s\le 1}\langle\eta^{m}_{\lceil mt\rceil}(s),\alpha_j^\vee\rangle\alpha_j,\alpha_j^\vee\rangle
\end{align*} 
\end{proof}

\begin{lem} \label{extrapol} It exists $C$ such that for any  $\epsilon>0$,  $m\ge 1$, $t\ge 0$, one has 
$$\P(\sum_{j=0  }^n\langle\frac{1}{m}\eta^{m}_{\lceil mt\rceil}(1)-\inf_{s\le 1}\langle\frac{1}{m}\eta^{m}_{\lceil mt\rceil}(s),\alpha_j^\vee\rangle\alpha_j,\alpha_j^\vee\rangle\ge \epsilon)\le \frac{C}{\epsilon^4m^2}$$
\end{lem}
\begin{proof} It comes from lemma \ref{conv2} and proposition \ref{maj}.
\end{proof}
\begin{prop} In the quotient space 
 $\{\frac{1}{m}\pi^{m}(mt): t\ge 0\}$ converges in a sense of finite dimensional law towards $\{B(t):t\ge 0\}$ when $m$ goes to infinity.\end{prop}
 \begin{proof}
 It comes from the convergence of $\{\frac{1}{m}\pi^{m}(\lceil mt\rceil): t\ge 0\}$, lemma \ref{extrapol} and proposition before.
 \end{proof}
 \begin{prop} In the quotient space, the sequence $\{\frac{1}{m}\pi^{m}(mt),t\ge 0\}$, $m\ge 1$,  is  tight. 
\end{prop}
\begin{proof} It is enough to prove that for any $i\in \{0,\dots,n\}$, $\epsilon,\eta>0$,  it exists an integer $k$ such that
\begin{align}\label{limsup}
\lim_m\sup k\max_{0\le l\le k-1}\P(\sup_{0\le r\le 1/k}\frac{1}{m}\langle\pi^{m}((r+l/k)m)-\pi^{m}(lm/k),\alpha_i^\vee\rangle\ge \eta)\le \epsilon.\end{align}
We write that $\vert \langle \pi^{m}(ms)-\pi^{m}(mt),\alpha_i^\vee\rangle\vert  $ is smaller than
 \begin{align*}
\vert\langle\pi^{m}(ms)-\pi^{m}(\lceil ms\rceil),\alpha_i^\vee\rangle\vert+ \vert\langle \pi^{m}(\lceil & ms\rceil )-\pi^{m}(\lceil mt\rceil),\alpha_i^\vee\rangle\vert\\
 &+ \vert\langle \pi^{m}(mt)-\pi^{m}(\lceil mt\rceil),\alpha_i^\vee\rangle\vert. \end{align*}
One has for $m,k\ge 1$, and $l\in\{0,\dots, k-1\}$
\begin{align*}
\P(\sup_{r\le 1/k}&\frac{1}{m}\langle\pi^{m}((r+l/k)m)-\pi^{m}(\lceil(r+l/k)m\rceil),\alpha_i^\vee\rangle\ge \eta)\\
&\le \P(\sup_{r\le 1/k}\frac{1}{m} \sum_{j=0  }^n\langle\eta^{m}_{\lceil (r+l/k)m\rceil}(1)-\inf_{s\le 1}\langle\eta^{m}_{\lceil (r+l/k)m\rceil}(s),\alpha_j^\vee\rangle\alpha_j,\alpha_j^\vee\rangle\ge \eta)\\
&\le\lceil \frac{m}{k}\rceil \P( \sum_{j=0  }^n\frac{1}{m}\langle\eta^{m}_{1}(1)-\inf_{s\le 1}\langle\eta^{m}_{1}(s),\alpha_j^\vee\rangle\alpha_j,\alpha_j^\vee\rangle\ge \eta)\le \lceil \frac{m}{k}\rceil\frac{C}{\eta^4m^2}.
\end{align*}
 We prove in a standard way that $\{\frac{1}{m}\pi^{m}(\lceil mt\rceil):t\ge 0\}$ satisfies (\ref{limsup}) for a particular integer $k$, which achieves the prove.
\end{proof}
 Thanks to the  Skorokhod representation theorem we can always suppose and we suppose  that the first convergence in propositon \ref{conv} is a locally uniform almost sure one. We have now all the ingredients to obtain the following theorem.  
 \begin{theo}
For every $t\ge 0$, and any sequence  $(i_k)_k$ of integers in $\{0,\dots,n\}$,
$$\frac{1}{m}\mathcal P_{\alpha_{i_k}}\dots \mathcal P_{\alpha_{i_0}}\pi^{m}(mt)\textrm{ and } \frac{1}{m}\mathcal P_{\alpha_{i_k}}\dots \mathcal P_{\alpha_{i_0}}\pi^{m}(\lceil mt\rceil),$$
converge almost surely towards $\mathcal P_{\alpha_{i_k}}\dots \mathcal P_{\alpha_{i_0}}B(t).$
\end{theo}
The theorem proves in particular that  the string coordinates associated to the Littelmann path model converges towards their analogs defined for the Brownian paths.  Actually  for any $t\ge 0$, if we consider the random sequence $(x^{m}_k(t))_k$ defined by
$$\mathcal P_{\alpha_{i_k}}\dots\mathcal P_{\alpha_{i_0}}\pi^{m}(t)=\pi^{m}(t)+\sum_{l=0}^kx^{m}_l(t)\alpha_{i_l},$$
and $(x_k(t))_k$ defined by 
$$\mathcal P_{\alpha_{i_k}}\dots\mathcal P_{\alpha_{i_0}}B(t)=B(t)+\sum_{l=0}^kx_l(t)\alpha_{i_l},$$
then the previous theorem shows that for every $k\ge 0$ and $t\ge 0$
$$\lim_{m\to \infty}\frac{1}{m}x_k^{m}(mt)=\lim_{m\to \infty}\frac{1}{m}x_k^{m}(\lceil mt\rceil)=x_k(t).$$ 
We can prove that this convergence remains true in law for $t=\infty$ provided that $\langle \widehat \nu,\alpha_i^\vee\rangle >0$ for every $i\in \{0,\dots,n\}$. In that case, one has the following convergence, which is proved in the appendix. 
 \begin{prop} \label{conv-string-infini} For every $k\ge 0$, the sequence  $(\frac{1}{m} x_k^{m}(\infty))_{m\ge 1}$ converges in law towards $x_k(\infty)$ when $m$ goes to infinity.
 \end{prop}
 \section{What we do not know}\label{WWDNK}
 We have proved in \cite{boubou-defo} that when $n=1$,
 $$\lim_{m\to\infty} \lim_{k\to \infty}\mathcal P_{\alpha_{i_k}}\dots\mathcal P_{\alpha_{i_0}} \frac{1}{m}\pi^{m}(mt)=\lim_{k\to\infty} \lim_{m\to \infty}\mathcal P_{\alpha_{i_k}}\dots\mathcal P_{\alpha_{i_0}} \frac{1}{m}\pi^{m}(mt),$$
 is not true as the  righthand side limit in $k$  doesn't even exist. Nevertheless we have proved that this identity becomes true if we replace the last Pitman transformation $\mathcal P_{\alpha_{i_k}}$ by a modified one which is a L\'evy transformation $\mathcal L_{\alpha_{i_k}}$. 
   We would like to show that a similar result exists   for $A_n^{1}$, with $n$ greater than $1$, but we didn't manage to get it for the moment. 
   Before saying what the correction could be, it is  with no doubt interesting to compare  graphically  the curves obtained applying successively Pitman transformations to a simulation of a Brownian curve with the ones obtained when at each stage these transformations are followed with a L\'evy transformation. On the picture \ref{BLP} the paths of the first sequence of  paths are represented in blue, whereas those of the second sequence are represented in yellow.  The red curve is the image of the simulation of the Brownian curve (which is a piecewise linear curve) by $\mathcal P$. We notice that there is an explosion phenomenon for the blue curves which doesn't exist for the yellow ones. 
    \section{What we hope}\label{WWH}
   Let us  try now to guess what the L\'evy transforms could become for $n$ greater than $1$.  For  an integral   path $\pi$ defined on $\R_+$  such that for all $i\in\{0,\dots,n\}$, $$\lim_{t\to \infty} \langle \pi(t),\alpha_{i}^\vee\rangle=+\infty,$$   the string coordinates for  $T=\infty$ are well defined, and if we denote them by   $\underline a$, one can let  and we let $\omega(\pi)=\omega(\underline a)$,  where $\omega$ is defined in proposition \ref{Cones}. Suppose now that  $\widehat \nu$ in (\ref{loi-discrete}) satisfies $\langle \widehat \nu,\alpha_i^\vee\rangle >0$ for all  $i\in \{0,\dots,n\}$. In this case, the random string coordinates $x_k^m(\infty)$, $k\ge 0$, and $\omega(\pi^{m})$ are well defined. The law of $(x_k^m(\infty))_{k\ge 0}$ is the
   probability measure $\upsilon^m$ on $B(\infty)$ defined by 
$$\upsilon^m(\underline{a})=C_me^{-\frac{1}{m}(\omega(\underline a))\vert \widehat\nu)}, \quad \underline{a}\in B(\infty),$$
where $$C_m=\prod_{\alpha\in\widehat R_+}(1-e^{-\frac{1}{m}(\alpha,\widehat\nu)}),$$
For any $t\in \N$, the law of  $x_k^m(t)$, $k\ge 0$, given that $\mathcal P\pi^m(t)=\lambda$ is the law of $x_k^m(\infty)$, $k\ge 0$, given that $(x_k^m(\infty))_{k\ge 0}$ belongs  to $B(\lambda)$ i.e. 
   $$ \forall p \ge 1,\,\,   \langle   \omega(\pi^m)-\sum_{k=0}^{p-1} x_k^m(\infty)\alpha_{i_k}-\frac{1}{2}x^m_{p}(\infty)\alpha_{i_p},\alpha_{i_p}^\vee\rangle\le \langle \lambda,\alpha_{i_p}^\vee\rangle.$$ 
   The random variable $\omega(\pi^{m})$ is distributed as $$\sum_{\alpha\in \widehat R_+}\mathcal G_\alpha\alpha,$$
where $ \mathcal G_\alpha, \alpha\in \widehat R_+,$ is a sequence of independent random variables such that $\mathcal G_\alpha$ has a geometric law with parameter $e^{-\frac{1}{m}(\alpha,\widehat\nu)}$. Thus, in the quotient space $\widehat{\mathfrak{h}}^*_\R/\R\delta$, when $m$ goes to infinity, $\frac{1}{m}\omega(\pi^{m})$ converges in law towards  

\begin{align}\label{weightexpo}
\sum_{\beta\in R_+} \mathcal E_\beta\beta+\sum_{\beta\in R_+} \sum_{k\ge 1}(\mathcal E_{\beta+k\delta}-\mathcal E_{-\beta+k\delta})\beta,
\end{align}
where $ \mathcal E_\alpha,$ $\alpha\in \widehat R_+$, are independent exponentially distributed random variables with parameters $(\widehat \nu,\alpha)$, $ \alpha\in \widehat R_+$.
  If the convergence were an almost sure one, denoting $\omega(B)$ the limit (which is not at this stage a function of $B$, but a random variable  which has to be heuristically thought as the weight of $B$ in a Verma module), the quantities  $$\frac{1}{m}\langle\omega(\pi^{m})-\sum_{k=0}^{p-1}x^{m}_k(\infty)\alpha_{i_k}-\frac{1}{2}x^{m}_p(\infty)\alpha_{i_p},\alpha_{i_p}^\vee\rangle, \, \, p \ge 1,$$ would converge almost surely towards
    $$ \langle\omega(B)-\sum_{k=0}^{p-1}x_k(\infty)\alpha_{i_k}-\frac{1}{2}x_p(\infty)\alpha_{i_p},\alpha_{i_p}^\vee\rangle,\, \, p \ge 1,$$ 
    when $m$ goes to infinity.
In the case when $n=1$, it exists a random variable $  \omega(B)$ distributed as (\ref{weightexpo}) such that
   $$\langle  \omega(B)-\sum_{k=0}^{p-1}x_k(\infty)\alpha_{i_k}-\frac{1}{2}x_{p}(\infty)\alpha_{i_p},\alpha_{i_p}^\vee\rangle$$
   converges to  $0$ when $p$ goes to infinity, which is    essential for the proofs, because of the inequalities defining  $B(\lambda)$. It actually allows to prove that all the convergences obtained for the string coordinates in  a continuous analog of $B(\infty)$ remains true when string coordinates are conditioned to belong to  a continuous analog of $B(\lambda)$. So we are going to suppose that this convergence remains true for   $n\ge 2$. 
   \begin{hyp}\label{hyp1} We suppose that  it exists   $ \omega(B)$, such that
   $$\langle  \omega(B)-\sum_{k=0}^{p-1}x_k(\infty)\alpha_{i_k}-\frac{1}{2}x_{p}(\infty)\alpha_{i_p},\alpha_{i_p}^\vee\rangle$$
   converges almost surely towards $0$ when $p$ goes to infinity.
   \end{hyp} 
   Moreover the following assumptions are supposed to be true. They are natural if we think that a Pitman--L\'evy type theorem exists for $n\ge 2$, which is for us a strong hope even if we don't have for the moment any piece of strong evidence which could allow to claim that it is a conjecture. From now on, we suppose that the sequence $(i_k)$ is periodic with period $n+1$. We can certainly release this hypothesis but there is no point to try to guess as general a result as possible.
     \begin{hyps}\label{hyp2}  $ $ 
   \begin{enumerate} 
   \item It exists a sequence  $(u_p)$ with values in $\mathfrak h^*_\R$  such that  
  $$\sum_{k=0}^{p-1}x_k(\infty)\alpha_{i_k}+u_p,$$
  converges almost surely in the   quotient  space $\widehat{\mathfrak h}_\R^*/\R\delta$ towards   $\omega(B)$ when $p$ goes to $+\infty$.
  \item The sequence $(x_p(\infty))_p$  converges almost surely towards $l\in \R$ as $p$ goes to infinity.  
  \item $\sum_{k=0}^{p}(x_k(\infty)-l)\alpha_{i_k}$, or equivalently (under assumption (2)), $\sum_{k=0}^{p}(x_k(\infty)-x_p(\infty))\alpha_{i_k}$, converges in the quotient space when $p$ goes to infinity.
   \end{enumerate}
   \end{hyps}
   \begin{rem}  In the case when $n=1$, one has $\lim_{p\to\infty}x_p(\infty)_p=2$ and one can take for instance $u_p=\frac{1}{2}x_p(\infty)\alpha_{i_p}$, or $u_p=\alpha_{i_p}$,  $p\ge 0$.
   \end{rem} 
   \noindent
      Under assumptions  \ref{hyp2},  if $l\ne 0$, it exists $u\in \mathfrak{h}_\R^*$  such that 
      $$\lim_{p\to \infty} u_p+x_p(\infty)\sum_{k=0}^{p-1}\alpha_{i_k}=  \lim_{p\to \infty} ux_p(\infty),$$
  Under assumptions \ref{hyp1} and \ref{hyp2}, $u$ must satisfy
  $$\forall p\ge 0,  \, \langle u-\sum_{k=0}^{p-1}\alpha_{i_k}- \frac{1}{2}\alpha_{i_p},\alpha_{i_p}^\vee\rangle=0.$$
It exists only one such a $u$ in $\mathfrak{h}^*$. 
 \begin{center}
 \fbox{ 
\begin{minipage}[c]{12.5cm}
The hope is that for all $t\ge 0$, almost surely,
 $$B(t)-\sum_{k=0}^{p-1}x_k(t)\alpha_{i_k}-x_p(t)(u-\sum_{i=0}^{p-1}\alpha_{i_k}),$$
 converges in the  quotient  space towards  $A(t)$ when $p$ goes to infinity.
\end{minipage}
}
\end{center}

\bigskip

For $n=1$, and $\mathfrak i=0,1,0\dots$, one can take $u=\alpha_0/2$, which gives the proper correction. 
  \begin{figure}
     \begin{center}
            \includegraphics[scale=0.55]{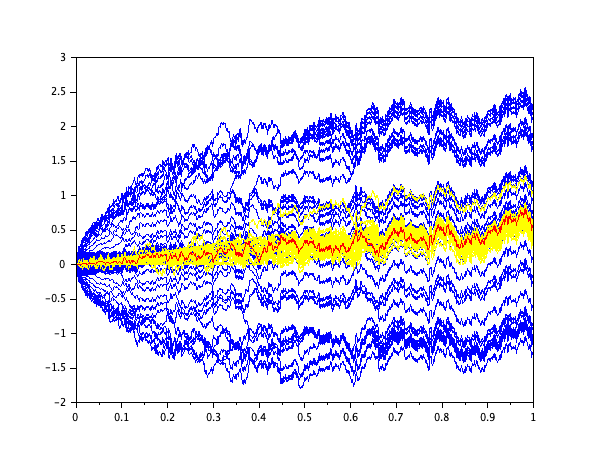}
            \caption{Successive transformations of a Brownian curve for $A_1^1$}
            \label{BLP}
                    \end{center} 
 \end{figure}
 
  \section{The case of $A_2^{(1)}$}
 We explicit the expected correction in the case when $n=2$, and $\mathfrak i=0,1,2,0,1, 2,\dots$. In that case one obtains that $u=\frac{1}{3}(\alpha_0-\alpha_2)$. 
  For $\eta(t)=t\Lambda_0+f(t), t\ge 0$,  where $f:\R_+\to \R\alpha_1+\R\alpha_2$ is a continuous fonction starting from $0$, we let for $i\in\{0,1,2\}$ 
$$\mathcal P_i\eta(t)=\eta(t)-\inf_{s\le t} \langle\eta(s),\alpha^\vee_i\rangle\alpha_i  ,$$
$$\mathcal L_i\eta(t)=\eta(t)-\frac{1}{3}\inf_{s\le t}\langle\eta(s),\alpha^\vee_i\rangle(\alpha_i-\alpha_{i+2}),$$
where $\alpha_3=\alpha_0$.

 \begin{center}
 \fbox{ 
\begin{minipage}[c]{12.5cm}
We hope that in the quotient space, for all $t\ge 0$, almost surely,
 $$\lim_k\mathcal L_{k}\mathcal P_{k-1}\dots \mathcal P_0B(t)=A(t),$$
 where the subcripts in the Pitman and the L\'evy   transformations must be taken modulo $3$.
\end{minipage}
}
\bigskip
\end{center}

 In figures  \ref{BLP1} and \ref{BLP2}  we have represented  successive transformations of a  simulated  brownian curve (evaluating on $\alpha_1^\vee$ and $\alpha_2^\vee$), with a correction and with no correction, similarly as in  figure  \ref{BLP}. We notice that the explosion phenomenon persists when there is no correction whereas it desapears with the expected needed correction. It gives some hope that  the ''conjecture'' is true.

  \begin{figure}
     \begin{center}
            \includegraphics[scale=0.55]{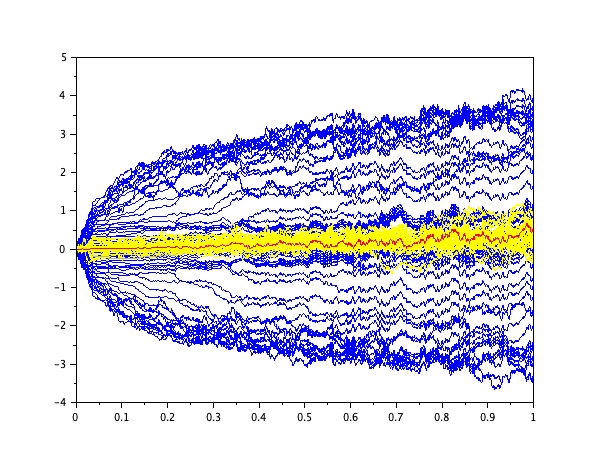}
            \caption{Successive transformations of a Brownian curve on  $\alpha_1^\vee$}
            \label{BLP1}
                    \end{center} 
 \end{figure}
 
  \begin{figure}
     \begin{center}
            \includegraphics[scale=0.55]{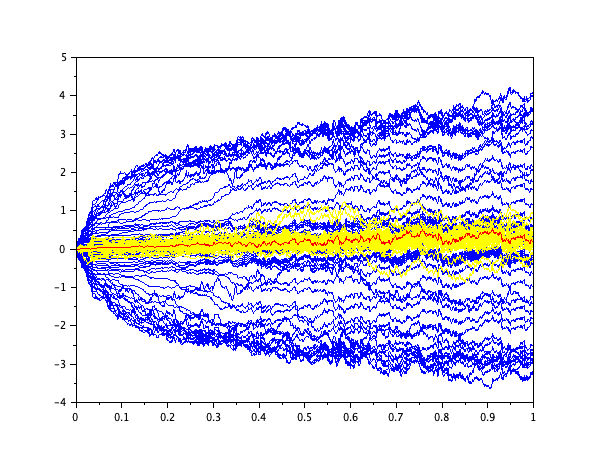}
            \caption{Successive transformations of a Brownian curve on   $\alpha_2^\vee$}
            \label{BLP2}
                    \end{center} 
 \end{figure}
 
 \section{Appendix}
 \begin{lem} Let $S_k^m=\sum_{i=1}^kX_i^{m}$, where $X_i^m$, $i\ge 0$, are independent identically distributed random variables such that 
 the Laplace transform of $X_1^m$ converges to the one of $X$ when $m$ goes to infinity. We suppose that $\E(X)>0$. Then  for any real $a$ such that $0<a<\E(X)$ 
  it exists $\theta>0, m_0\ge 0$ such that for all $m\ge m_0$, $k\ge 0$, $$\P(S_k^m<ka)\le  e^{-k \theta}.$$ 
 
 \end{lem}
 \begin{proof}
  Let $\rho=\E(X)$, $c=\rho-a.$ Choose $\lambda>0$ such that $\E(e^{-\lambda(X-\rho)})\le e^{\lambda c/4}$. After that we choose $m_0$ such that for $m\ge m_0$, $$\E( e^{-\lambda (X^m_1-\rho)})\le e^{\lambda c/4}\E(e^{-\lambda (X-\rho)})\le e^{\lambda c/2}.$$ One has 
 $$\P(S_k^m<k(\rho-c))=\P(e^{-\lambda S_k^m}>e^{-\lambda k(\rho-c)})\le (e^{\lambda(\rho-c)}\E(e^{-\lambda X_1^m}))^k\le e^{-k \lambda c/2}.$$
 
 \end{proof}
  
 \begin{proof}[Proof of proposition \ref{conv-string-infini}]  Let us prove that it is true for $k=0$. Proposition will follow by induction. We have seen that for any $T\ge0$, almost surely 
 $$\lim_m\inf_{t\le T}\frac{1}{m}\langle \pi^{m}(mt),\alpha_{i_0}^\vee\rangle=\inf_{t\le T}\langle B(t),\alpha_{i_0}^\vee\rangle.$$
 Suppose that $\langle \widehat \nu,\alpha_{i_0}^\vee\rangle >0$. Then  $\inf_{t\ge 0}\langle B(t),\alpha_{i_0}^\vee\rangle>-\infty$, and we want to prove that 
 $$\lim_m\inf_{t\ge 0}\frac{1}{m}\langle \pi^{m}(mt),\alpha_{i_0}^\vee\rangle=\inf_{t\ge 0}\langle B(t),\alpha_{i_0}^\vee\rangle.$$
 It is enough to prove that for any $\epsilon>0$ it exists $T,m_0\ge 0$ such that for any $m\ge m_0$,
 $$\P(\inf_{t\ge T}\pi^m(mt)<0)\le \epsilon.$$
 For this let $$S_k^m=\sum_{i=1}^k\frac{1}{m+1}\sum_{j=1}^m\langle \eta_{(i-1)m+j}^m(1),\alpha_{i_0}^\vee\rangle.$$ It satisfies the hypothesis of the previous lemma, with $$X_i^m=\frac{1}{m+1}\sum_{j=1}^m\langle \eta_{(i-1)m+j}^m(1),\alpha_{i_0}^\vee\rangle,$$
 which converges in law towards  $\langle B_1,\alpha_{i_0}^\vee\rangle$ as $m$ goes to infinity. One has $$S_k^m=\frac{1}{m+1}\sum_{i=1}^{mk}\langle \eta_{i}^m(1),\alpha_{i_0}^\vee\rangle=\frac{1}{m+1}\pi^m(mk).$$ We let $\rho=\E(\langle B_1,\alpha_{i_0}^\vee\rangle)=\langle \widehat \nu,\alpha_{i_0}^\vee\rangle$. Let $\epsilon >0$,   $0<a<b<\rho$. As $ka+(b-a)\sqrt{k}\le b k$, forall $k\ge 1$, we choose $\theta>0$, $m_0\ge 0,$ such that  for all $k\ge 0$, $m\ge m_0$,
 \begin{align}\label{cramer}
 \P(S_k^m<ka+(b-a)\sqrt{k})\le e^{-k\theta},
 \end{align}
 i.e. 
  $$\P(\langle \pi^m(mk),\alpha_{i_0}^\vee\rangle<a(m+1)k+(b-a)(m+1)\sqrt{k})\le e^{-k \theta}.$$
 One has for $N\in \N^*$
 $$\{\inf_{t\ge N}\frac{1}{m}\langle\pi^{m}(\lceil mt\rceil),\alpha_{i_0}^\vee\rangle < at\}\subset \cup_{k\ge N}\cup_{0\le p\le m}\{\langle\pi^{m}(mk+p),\alpha_{i_0}^\vee\rangle\le a(mk+p)\},$$
 and 
 \begin{align*}
 \P(\inf_{t\ge N} \frac{1}{m}\langle\pi^{m}(&\lceil mt\rceil),\alpha_{i_0}^\vee\rangle < at)\le \P( \cup_{k\ge N}\{\langle\pi^m(mk),\alpha_{i_0}^\vee\rangle\le a(m+1)k+(b-a)(m+1)\sqrt{k}\})\\
 &+
 \P(\cup_{k\ge N}\{\sup_{0\le p\le m} \vert \langle\pi^{m}(mk+p),\alpha_{i_0}^\vee\rangle-\langle\pi^m(mk),\alpha_{i_0}^\vee\rangle\vert\ge (b-a)(m+1)\sqrt{k}\})
 \end{align*}
 Thanks to the lemma \ref{cramer} we can choose and we choose $N$ such that  the first probability is smaller than $\epsilon$. Besides
 \begin{align*}
 \P( \sup_{0\le p\le m} \vert \langle\pi^{m}&(mk+p),\alpha_{i_0}^\vee\rangle-\langle\pi^m(mk),\alpha_{i_0}^\vee\rangle\vert\ge (b-a)(m+1)\sqrt{k})\\
 &\le\frac{1}{(b-a)^4(m+1)^4k^2} \E( \sup_{0\le p\le m} \vert\langle\pi^{m}(mk+p),\alpha_{i_0}^\vee\rangle-\langle\pi^m(mk),\alpha_{i_0}^\vee\rangle\vert^4)
 \end{align*}
 One has 
 \begin{align*}
 \pi^m(mk+p)&=\sum_{i=1}^{mk+p}\langle \eta_i^m(1),\alpha_{i_0}^\vee\rangle-\E\langle \eta_i^m(1),\alpha_{i_0}^\vee\rangle+(mk+p)\E\langle \eta_1^m(1),\alpha_{i_0}^\vee\rangle\\
 &=Y^m(mk+p)+(mk+p)\E\langle \eta_1^m(1),\alpha_{i_0}^\vee\rangle,
 \end{align*}
 where $Y^m(mk+p)$, $p\in\{0,\dots,m\}$, is a martingale.
 A maximale inequality and the fact that $\E\langle\eta_1^m(1),\alpha_{i_0}^\vee\rangle\sim \langle\widehat \nu,\alpha_{i_0}^\vee\rangle$ imply that it exists $C$ such that 
 $$
 \P( \sup_{0\le p\le m} \vert \langle\pi^{m}(mk+p),\alpha_{i_0}^\vee\rangle-\langle\pi^m(mk),\alpha_{i_0}^\vee\rangle\vert\ge (b-a)(m+1)\sqrt{k})\\\le C/k^2.$$
 Since proposition \ref{boundinc} and lemma \ref{extrapol}    ensure that $\pi^{m}(mt)-\pi^m(\lceil mt\rceil)$ is bounded by a random variable $\xi_{\lceil mt\rceil}^m$ satisfying for any $u>0$
 $$\P(\xi_k^m/m\ge u\sqrt{k})\le \frac{\widetilde C}{k^2m^2},$$
 one obtains the proposition. 
 \end{proof}

\end{document}